\documentclass[12pt,oneside]{amsart}
\usepackage{amsfonts, amsthm, amsmath, amssymb, ,mathtools} 
\usepackage[backend=biber, style=alphabetic]{biblatex} 
\usepackage[colorlinks=true, citecolor=blue, backref=none]{hyperref}  
\usepackage{bbm} 
\usepackage{url}
\usepackage[hyphenbreaks]{breakurl}
\usepackage{tikz-cd} 
\usepackage[scr=boondox,cal=cm]{mathalpha} 
\usepackage[layout=a4paper,right=2cm,left=2cm]{geometry}
\numberwithin{equation}{section}
\usepackage{enumitem} 
\newcounter{dummy}
\makeatletter
\newcommand\myitem[1][]{\item[#1]\refstepcounter{dummy}\def\@currentlabel{#1}}
\makeatother
\bibliography{Caputo_AlgebraicProofBrauerFormula}

\theoremstyle{plain}
\newtheorem{theorem}{Theorem}[section]
\newtheorem{proposition}[theorem]{Proposition}
\newtheorem{corollary}[theorem]{Corollary}
\newtheorem{lemma}[theorem]{Lemma}
\newtheorem{definition}[theorem]{Definition}

\theoremstyle{remark}
\newtheorem{remark}[theorem]{Remark}
\newtheorem{example}[theorem]{Example}
\newtheorem*{notation}{Notation}
\newtheorem*{acknowledgment}{Acknowledgment}

\newcommand{\Z}{\mathbb{Z}} 
\newcommand{\Q}{\mathbb{Q}} 
\newcommand{\R}{\mathbb{R}} 
\newcommand{\Cl}[2]{Cl^{#2}_{#1}} 
\newcommand{\cl}[2]{h^{#2}_{#1}} 
\newcommand{\idunits}[2]{\mathcal{U}^{#2}_{#1}} 
\newcommand{\units}[2]{E_{#1}^{#2}} 
\newcommand{\locunits}[1]{U_{#1}} 
\newcommand{\idCl}[1]{\mathcal{C}_{#1}} 
\newcommand{\Qq}[2]{Q_{#1}^{#2}} 
\DeclarePairedDelimiterXPP{\hht}[3]{\widehat{h}^{#1}}{(}{)}{}{#2,#3} 
\DeclarePairedDelimiterXPP{\hh}[3]{h^{#1}}{(}{)}{}{#2,#3} 
\DeclarePairedDelimiterXPP{\HHT}[3]{\widehat{H}^{#1}}{(}{)}{}{#2,#3} 
\DeclarePairedDelimiterXPP{\HH}[3]{H^{#1}}{(}{)}{}{#2,#3} 
\newcommand{\Ext}[4]{\mathrm{Ext}^{#1}_{#2}\left(#3,#4\right)} 
\newcommand{\Hom}{\mathrm{Hom}} 
\newcommand{\Ind}{\mathrm{Ind}} 
\newcommand{\Inf}{\mathrm{Inf}} 
\newcommand{\Res}{\mathrm{Res}} 
\newcommand{\Ker}{\mathrm{Ker}} 
\newcommand{\Coker}{\mathrm{Coker}} 
\newcommand{\dual}[1]{#1^*} 
\newcommand{\WW}[1]{\mathcal{W}^{#1}} 
\newcommand{\tors}[1]{#1_{\mathrm{tors}}} 
\newcommand{\mt}[1]{\overline{#1}} 
\newcommand{\rc}[3][]{\mathcal{C}^{#1}_{#2}\!\left(#3\right)} 
\newcommand{\card}[1]{\left|#1\right|} 
 
\newcommand{\kd}[2]{\psi_{#1}\left(#2\right)} 
 
\newcommand{\roots}[1]{\mu_{#1}} 
\newcommand{\reg}[2]{R_{#1}^{#2}} 
\newcommand{\rcring}{\mathscr{R}}
\newcommand{\rcfield}{\mathscr{K}}

\newcommand{\rcringalt}{\mathscr{S}}
\newcommand{\rcfieldalt}{\mathscr{E}}
\newcommand{\rcfieldext}{\mathscr{L}}
\newcommand{\cont}[1]{\chi(#1)} 
\newcommand{\ie}{i.~e.~} 

\begin{document}
\title{\textsc{An algebraic proof of Brauer's class number formula}}
\author{Luca Caputo}
\email{luca.caputo@gmx.com}
\subjclass[2020]{11R29, 11R33, 11R34, 11R37}
\begin{abstract}
We show how, starting from the Tate exact sequence for units of Ritter and Weiss, one can obtain an algebraic proof of Brauer's class number formula, using the formalism of regulator constants. 
\end{abstract}

\maketitle


\section{Introduction}
Let $K/k$ be a finite Galois extension of number fields with Galois group $G$ and let $S$ be a finite $G$-stable set of primes of $K$ containing the archimedean primes. If $G$ is not cyclic, one can find a non empty set of non zero integers $\left\{n_H\right\}_{H\leq G}$ such that the following formula holds
\begin{equation}\label{eq:bcnf}
\prod_{H\leq G}\left(\cl{K^H}{S}\right)^{n_H}= \prod_{H\leq G}\left(\frac{\card{\roots{K^H}}}{\reg{K^H}{S}}\right)^{n_H}.
\end{equation}
Here $\cl{K^H}{S}$ is the $S$-class number of $K$, \ie the order of the $S_{K^H}$-class group of $K^H$ (where $S_{K^H}$ is the set of primes of $K^H$ which lie below primes of $S$), $\reg{K^H}{S}$ is the regulator of the $S_{K^H}$-units of $K^H$ and $\roots{K^H}$ is the group of roots of unity of $K^H$. 

This result was proved for the first time in generality by Brauer and Kuroda independently (see \cite{Brauer1951} and \cite{Kuroda1950}). More precisely, they showed that, if there is an isomorphism of $\Q[G]$-modules 
\begin{equation}\label{eq:br}
\bigoplus_{n_H> 0}\Q[G/H]^{n_H} \cong \bigoplus_{n_H< 0}\Q[G/H]^{-n_H},
\end{equation}
then the above formula holds, as a consequence of Artin’s formalism of $L$-functions and the formula for the residue of the Dedekind zeta function at $s=1$. A set $\left\{n_H\right\}_{H\leq G}$ satisfying \eqref{eq:br} is referred to as a Brauer relation, briefly denoted as $\sum_{H\leq G}n_H H$, and formula \eqref{eq:bcnf} is called Brauer's class number formula. One can show that every non cyclic group has a Brauer relation and cyclic groups have none: for the most general results on Brauer relations, see \cite{Tornehave1984}, \cite{Bouc2006}, \cite{BartelDokchitser2015} and \cite{BartelDokchitser2014}. Brauer also showed that the left-hand side of \eqref{eq:bcnf} can assume only finitely many values as $K$ varies over the Galois extensions of $k=\Q$ with Galois group $G$ and $\card{S}$ is fixed.

It is natural to look for an algebraic proof of \eqref{eq:bcnf} and specifically one describing the right-hand term of \eqref{eq:bcnf} explicitly as a rational number. In fact the search for such proof has a long history. The first particular case to be analyzed was that where $G$ is the smallest non cyclic group, namely $G=\Z/2 \times \Z/2$. In such case, assuming that $K=\Q(\sqrt{d},\sqrt{-d})$ for some square-free integer $d>1$, Hilbert gave an algebraic proof of the following formula (see \cite{Hilbert1894}) which was first proved analytically by Dirichlet\footnote{About this result, Dirichlet wrote (see \cite[p. 378]{Dirichlet1841}, translated from German by C. Greither): \textit{If we do not go wrong entirely, this result contains one of the most beautiful theorems in the theory of complex numbers. And it is bound to surprise us all the more, considering that in the real theory, there is no apparent connection at all between forms that correspond to opposite discriminants.}
}
\[
h_K = \frac{1}{2}Q_Kh_{\Q(\sqrt{d})}h_{\Q(\sqrt{-d})}.
\]
Here $S$ is implicitly assumed to be the set of archimedean primes (and is omitted from the notation) and $Q_K=1,2$ depending on certain explicit conditions on the norm of fundamental units or roots of unity of $K$. 
Hilbert's result was later generalized by various authors (see \cite{Lemmermeyer1994}, where a detailed bibliography as well as the most general proof can be found). Other particular cases analyzed in details include elementary abelian $p$-groups and dihedral groups (see for instance \cite{Kani1994}, \cite{deSmit2001}, \cite{BarteldeSmit2013}, \cite{CaputoNuccio2019} and the references therein). To our knowledge, the only algebraic proof of \eqref{eq:bcnf} which holds with no restriction on the finite group $G$, the ground number field $k$ or the finite $G$-invariant set of primes $S$ (containing the archimedean primes) is the one given by de Smit (see \cite{deSmit2002}).

Hilbert's proof already suggests a general approach to obtain such proofs: on the one hand the ratio of class numbers is shown to be equal to an index of groups of units of subfields of $K$ in the group of units of $K$. On the other hand, the unit index can be translated into a ratio of regulators and orders of groups of roots of unity. This translation step is carried out in general by Brauer (see \cite{Brauer1951}, \cite{Walter1979}, \cite{Kani1994}), while some of the most recent papers omit it and focus rather on finding general methods to describe the unit index (see \cite{deSmit2001}, \cite{BosmaDeSmit2001}, \cite{Bartel2012}, \cite{BarteldeSmit2013}, \cite{CaputoNuccio2019}). Such description of the unit index may then be used to compute explicit bounds on the left-hand side of \eqref{eq:bcnf}.  

In this paper we propose a new algebraic proof of \eqref{eq:bcnf}, which holds for any finite group $G$, any ground number field $k$ and any finite set of primes $S$ of $K$ (containing the archimedean primes). Here, contrarily to de Smit's proof which treats units and class groups in an integrated way, we follow somehow Hilbert's strategy and separate units from class groups. The pivotal role of the unit index is played here by the so-called regulator constant of the group of units of $K$ associated to the Brauer relation $\sum_{H\leq G}n_H H$. The regulator constant of a finitely generated $\Z[G]$-module $M$ associated to a relation $\sum_{H\leq G}n_H H$ is essentially a ratio of determinants of a bilinear form on submodules of fixed elements $M^H$, raised to the power $n_H$, as $H$ varies over the subgroups of $G$ (see Section \ref{sec:rc} for a precise definition). The module $M$ is usually assumed to be without $\Z$-torsion or rather the torsion-free quotient of $M$ is used instead (so a module has the same regulator constant as its torsion free quotient). Since we constantly need to deal with modules with torsion, we will introduce a definition of regulator constants which is slightly different from the one in the literature: in particular we will add a term accounting explicitly for the torsion of the module under consideration, which is a ratio of orders of $\tors{M^H}$.  

Regulator constants have been introduced by T. and V. Dokchitser (see \cite{DokchitserDokchitser2009}) and later also used in a series of papers by other authors (see \cite{Bartel2012}, \cite{BarteldeSmit2013}, \cite{Bartel2014}). However, a closely related invariant was introduced some years before by Kani (see \cite{Kani1994}), but this work seems unfortunately to have gone unnoticed. Kani proved, among many other properties of these invariants, that, although they are not multiplicative in exact sequences, the defect from being multiplicative can be explicitly described in terms of cohomology. Moreover he also proved that these invariants are trivial for cohomologically trivial modules. We will give a more direct proof of these results for regulator constants in Section \ref{sec:rc}. Here our modified definition of regulator constants allows us to get neater statements.

The strategy of the proof of \eqref{eq:bcnf} presented in this paper is the following. On the one hand, the regulator constants is related to the right-hand term of \eqref{eq:bcnf}, as shown by Kani and Bartel independently (see Theorem \ref{thm:rcreg}). Bartel's proof is very direct and is essentially a linear algebra computation. On the other hand, we prove that regulator constants are also related to the left-hand side of \eqref{eq:bcnf} (see Theorem \ref{thm:rccln}). To do so, we consider the Tate exact sequence for units introduced by Ritter and Weiss (see \cite{RitterWeiss1996}), which generalizes to any finite set of primes $S$ the result obtained of Tate for $S$ large (see \cite{Tate1966}). More precisely, Ritter and Weiss showed that there exist two exact sequences of finitely generated $\Z[G]$-modules
\[
0\to \units{K}{S} \to A \to B \to \nabla\to 0
\]
\[
0\to \Cl{K}{S} \to \nabla \to \Z[S]\oplus \WW{S} \to \Z\to 0
\]
where $\units{K}{S}$ is the groups of $S$-units of $K$, $\Cl{K}{S}$ is the $S$-class groups of $K$, $A$ and $B$ are cohomologically trivial $\Z[G]$-module and $\WW{S}$ is an explicit semilocal $\Z[G]$-module whose exact definition will be recalled in Section \ref{sec:tes}. Using the properties of regulators constants mentioned above, we obtain a formula relating the regulator constants of units to the regulator constants of the class group and the latter is in turn related to the class numbers by the ambiguous class number formula. In fact the regulator constant of the class group differs from the ratio of class numbers by a cohomological term. This will cancel out with the regulator constants multiplicativity defect arising from the Ritter-Weiss exact sequences, allowing us to complete the proof of \eqref{eq:bcnf}. 
   
\begin{acknowledgment}
The author is very grateful to Filippo Nuccio for several useful suggestions on a preliminary version of the paper and to Cornelius Greither for the translation of Dirichlet's quote in the introduction.  
\end{acknowledgment}
\section{Regulator constants}\label{sec:rc}
In this section, after recalling the definition and the basic properties of regulator constants, we describe their behavior in exact sequences and show that regulator constants of cohomologically trivial modules are trivial. For the entire section, $G$ is a finite group, $\rcring$ is a principal ideal domain with field of fraction $\rcfield$.

\begin{definition}
A formal linear combination $\sum n_H H$ of subgroups of $G$ with coefficients in $\Z$ is a \emph{Brauer relation} in $G$ if the virtual permutation representation $\sum_H \Q[G/H]^{n_H}$ is zero. 
\end{definition}

\begin{notation}
Let $M$ be an $\rcring$-module. We write $\tors{M}$ for the submodule of torsion elements of $M$ and set $\mt{M}:=M/\tors{M}$. For any $\rcring$-homomorphism $f:M\to N$, we write $\mt{f}:\mt{M}\to \mt{N}$ for the induced homomorphism. 
\end{notation}

\begin{notation}
For a finitely generated torsion $\rcring$-module $M$, we denote by $\cont{M}$ the content of $M$ which is a fractional ideal of $\rcring$ (see \cite[Chap. VII, \S4, no. 5, Définition 4]{BourbakiAC}). This is the same as the product of the elementary divisors of $M$ or the product of the invariant factors of $M$. In particular, when $\rcring = \Z$, we have $\cont{M}=\card{M}\Z$, where $\card{M}$ is the order of $M$. Since every fractional ideal of $\rcring$ is principal, $\cont{M}$ can and will be identified with an element of $\rcfield^\times/\rcring^\times$. Note in particular that $\cont{M}^2\in \left(\rcfield^\times/\rcring^\times\right)^2 = \left(\rcfield^\times\right)^2/\left(\rcring^\times\right)^2\subseteq\rcfield^\times/\left(\rcring^\times\right)^2$. 
\end{notation}

\begin{notation}
Let $M$ be a finitely generated $\rcring[G]$-module. Then we set
\[
\dual{M}=\Hom_\rcring(M,\rcring)
\]  
for the $\rcring$-dual of $M$ which we regard as a $\rcring[G]$-module with the contragredient action. Moreover, if $H\leq G$, $M^H$ denotes the submodule of $M$ fixed by elements of $H$.  
\end{notation}

\begin{definition}\label{def:nd}
Let $M$ be a finitely generated $\rcring[G]$-module and let $\rcfieldext$ be an extension of $\rcfield$. We say that $M$ is self-dual over $\rcfieldext$ if there is a $\rcfieldext[G]$ isomorphism $M\otimes_\rcring \rcfieldext\cong \dual{M}\otimes_\rcring \rcfieldext$.
\end{definition}

\begin{remark}
One easily checks that being self-dual over $\rcfieldext$ is equivalent to the existence of $\rcring[G]$-homomorphism
\[
M \longrightarrow \dual{M} \otimes \rcfieldext = \Hom(M,\rcfieldext)
\] 
which is injective on $\mt{M}$. This is in turn equivalent to the existence of a $\rcring[G]$-bilinear pairing on $M$ with values in $\rcfieldext$ which is non degenerate on $\mt{M}$. 
\end{remark}

\begin{remark}
By the Noether--Deuring theorem (see \cite[Introduction, Exercise \S 6.6]{CurtisReiner1981}), $M$ is self-dual over $\rcfieldext$ if and only if it is self-dual over $\rcfield$. The reason for having $\rcfieldext$, instead of simply $\rcfield$, in the definition of a self-dual module is that in applications one considers $\Z[G]$-modules and the relevant pairings take value in $\rcfieldext = \R$ (see Example \ref{ex:rcexa}).

We will therefore avoid mentioning the field and simply say that a $\rcring[G]$-module $M$ is self-dual. This notion however should not be confused with the stronger notion of a unimodular $\rcring[G]$-module $M$, \ie one for which there exists an isomorphism $M\cong \dual{M}$.  
\end{remark}

\begin{remark}
A finitely generated $\Z[G]$-module $M$ is always self-dual over $\Q$ since the characters of $M~\otimes_\Z~\Q~$ and $\dual{M}~\otimes_\Z~\Q$ are conjugate, hence equal.  
\end{remark}

We will be using the following definition of regulator constants, which is a slight modification of the ones given in the literature (see \cite[Definition 2.13]{DokchitserDokchitser2009} \cite[Definition 2.3]{BarteldeSmit2013}, \cite[Definition 2.10]{Bartel2014}, \cite[Theorem 2.5]{Kani1994}). The term accounting for the torsion is introduced to get neater statements: for instance, with this definition, the regulator constants of a cohomologically trivial module are trivial. 

\begin{definition}\label{def:rc}
Assume that $\rcfield$ has characteristic coprime with $\card{G}$ and let $\rcfieldext$ be an extension of $\rcfield$. Let $M$ be a finitely generated $\rcring[G]$-module which is self-dual over $\rcfieldext$ and let $\langle\cdot,\cdot\rangle~:~M~\times~M~\to~\rcfieldext$ be a $\rcring[G]$-bilinear pairing that is non-degenerate on $\mt{M}$. Let $\Theta=\sum_{H\leq G} n_H H$ be a Brauer relation in $G$. Then the regulator constant $\rc{\Theta}{M}$ of $M$ with respect to $\Theta$ is
\[
\rc[\rcring]{\Theta}{M} = \prod_{H\leq G}\left( \frac{1}{\cont{\tors{M}^H}^2}\det\left(\frac{1}{\card{H}}\left.\langle\cdot,\cdot\rangle \right|_{\mt{M^H}} \right)\right)^{n_H} \in \rcfield^\times/\left(\rcring^\times\right)^2,
\]   
where the determinant appearing in the factor corresponding to the subgroup $H$ is evaluated on any $\rcring$-basis of $\mt{M^H}$.
\end{definition}

Some checks are needed to ensure regulator constants are well-defined. First of all, if the pairing is non degenerate on $\mt{M}$, then it is also non degenerate on $\mt{M^H}$ for any $H\leq G$ (see \cite[Lemma 2.15]{DokchitserDokchitser2009}). Moreover, the definition is independent of the choice of the pairing and the field extension $\rcfieldext$ (see \cite[Theorem 2.17]{DokchitserDokchitser2009}). In particular, since one can always choose a $\rcfield$-valued non degenerate pairing, $\rc[\rcring]{\Theta}{M}$ indeed lies in $\rcfield^\times/\left(\rcring^\times\right)^2$.

\begin{remark}
The definition of $\delta_\rcring(\Theta,M)$ in \cite{Kani1994} is very similar to the one of $\rc[\rcring]{\Theta}{M}$ given above:
\[
\delta_\rcring(\Theta,M) = \prod_{H\leq G}\left( \frac{1}{\cont{\tors{M}^H}}\det\left(\frac{1}{\card{H}}\left.\langle\cdot,\cdot\rangle \right|_{\mt{M^H}} \right)\right)^{n_H} \in \rcfield^\times/\rcring^\times
\] 
for any pairing $\langle\cdot, \cdot\rangle$ as in Definition \ref{def:rc} (see \cite[Theorem 2.5]{Kani1994}). The main difference is that in Definition \ref{def:rc} we have $\cont{\tors{M}^H}^2$ instead of $\cont{\tors{M}^H}$. Note that some of the results in \cite{Kani1994} are not correct in general if we use Kani's definition but become true using the one given here (see the remark before Proposition \ref{prop:rces}). Correcting Kani's definition with the square of the content, we see that $\delta_\rcring(\Theta, M)$ is just the image of $\rc[\rcring]{\Theta}{M}$ under the homomorphism $\rcfield^\times/\left(\rcring^\times\right)^2 \to \rcfield^\times/\rcring^\times$. 
\end{remark}

\begin{example}
It is immediate to check that free $\rcring[G]$-modules are self-dual and have trivial regulator constants (this also follows from a more general result on regulator constants of permutation modules, see \cite[Example 2.19]{DokchitserDokchitser2009}). 
\end{example}

\begin{example}\label{ex:rcexa}
A more interesting example is the following (see \cite[Theorem 8.1]{Kani1994}, \cite[Lemma 2.12]{Bartel2012}). Let $K/k$ be a Galois extension of number fields with Galois group $G$ and let $\units{K}{S}$ be the group of $S$-units of $K$ where $S$ is a finite set of primes of $K$ containing the archimedean primes. This is a $\Z[G]$-module and one can define a $G$-invariant pairing on $\units{K}{S}$ with values in $\R$ by
\[
\langle u,u'\rangle =\sum_{w\in S} d_{w}\log|u|_w\log|u'|_w \quad\textrm{for $u,u'\in\units{K}{S}$}
\]  
where $d_{w}=[K_w:\Q_v]$, $K_w$ is the completion of $K$ at $w$ and $\Q_v$ is the completion of $\Q$ at the valuation $v$ extended by $w$. The absolute value $|\cdot|_w$ is normalized in such a way that its restriction to $\Q$ coincides with the usual $p$-adic or archimedean absolute value of $\Q$. One shows that the determinant of this pairing (computed on a basis of $\mt{\units{K}{S}}$) equals the square of the $S$-regulator of $K$ times an explicit positive rational number. This already suggests that regulator constants of the groups of units should be related to regulators in general.  
\end{example}

\begin{remark}
Although we will not use the notion of factorisability here, we remark that, using our definition, Bartel's result relating factorisability to regulator constants just says that two finitely generated $\Z[G]$-modules $M$ and $N$ with $M\otimes \Q \cong N\otimes \Q$ are factor equivalent if and only if they have the same regulator constant for every Brauer relation in $G$ (see \cite[Corollary 2.12]{Bartel2014}). 
\end{remark}

We now list some of the main elementary properties of the regulator constants. We will need the following notation.

\begin{notation}
Let $\Theta = \sum_{H\leq G}n_H H$ be a Brauer relation in $G$. 
\begin{itemize}
\item If $G$ is a subgroup of a finite group $X$, then set  
\[
\Ind_{G}^{X}\Theta =  \sum_{H\leq G} n_H H
\]
where now the groups $H$ are regarded as subgroups of $X$. Then $\Ind_{G}^{X}\Theta$ is a Brauer relation in $G'$ (by transitivity of induction).
\item If $Y$ is a subgroup of $G$, then set 
\[
\Res_{Y}^{G}\Theta =  \sum_{H\leq G} \sum_{g\in H\backslash G/Y} n_H(Y\cap H^{g^{-1}})
\]
where, for an element $g\in G$ and a subgroup $H\leq G$, we set $H^g = gHg^{-1}$. Then $\Res_{Y}^{G}\Theta$ is a Brauer relation in $G$ (by Mackey's decomposition).
\item If $G$ is a quotient of a finite group $Z$, then set 
\[
\Inf_{G}^{Z}\Theta =  \sum_{H\leq G} n_H H'
\]
where $H'$ is the inverse image of $H$ under the quotient map $Z\to G$. Then $\Inf_{G}^{X}\Theta$ is a Brauer relation in $X$.
\end{itemize}
\end{notation}

\begin{proposition}\label{prop:rcprop}
Let $\Theta$ be a Brauer relation in $G$.
\begin{enumerate}[label = (\roman*)]
\item\label{prop:rcpropind} If $G$ is a subgroup of a group $X$ and $N$ is a finitely generated self-dual $\rcring[X]$-module, then 
\[
\rc[\rcring]{\Ind_{G}^{X}\Theta}{N} = \rc[\rcring]{\Theta}{\Res_G^{X} N}.
\]
\item\label{prop:rcpropres} If $Y$ is a subgroup of $G$ and $N$ is a finitely generated self-dual $\rcring[X]$-module, then 
\[
\rc[\rcring]{\Res_{Y}^{G}\Theta}{N} = \rc[\rcring]{\Theta}{\Ind_{Y}^G N}.
\]
\item\label{prop:rcpropinf} If $G$ is a quotient of a group $Z$ and $M$ is a finitely generated self-dual$\Z[G]$-module, then
\[
\rc[\rcring]{\Inf_G^{Z}\Theta}{\Inf_G^{Z}M} = \rc[\rcring]{\Theta}{M}. 
\]
\item\label{prop:rcpropext} If $\rcringalt$ is a principal ideal domain containing $\rcring$, with field of fraction $\rcfieldalt$ and $M$ is a finitely generated self-dual $\rcring[G]$-module, then $\rc[\rcringalt]{\Theta}{M\otimes_\rcring\rcringalt}$ is the image of $\rc[\rcring]{\Theta}{M}$ under the map $\rcfield^\times/\left(\rcring^\times\right)^2 \to \rcfieldalt^\times/\left(\rcringalt^\times\right)^2$.
\end{enumerate}
\end{proposition}
\begin{proof}
Since taking torsion commutes with restriction, induction, inflation and scalar extension ($\rcringalt$ is $\rcring$-flat), we can prove separately the assertions for the determinants factor and the content factor in the definition of regulator constants. For the determinants factor, the proof is given in \cite[Proposition 2.45]{DokchitserDokchitser2009} and \cite[Corollary 2.18]{DokchitserDokchitser2009}. 

Here we prove the assertions for the content factor and, by the initial remark, we can assume that all modules are $\rcring$-torsion. Assertions \ref{prop:rcpropind} and \ref{prop:rcpropinf} are immediate to check. For \ref{prop:rcpropres}, observe that by Mackey's decomposition we can write
\[
\Ind_{Y}^G N \cong \bigoplus_{g\in H\backslash G/Y} \Ind_{H\cap Y^g}^H N^g
\]	
as $\rcring[H]$-modules, where $N^g$ is the $\rcring[Y]$-module with the action of $Y$ defined by $y\cdot n = (g^{-1}yg)n$ for $y\in Y$ and $n\in N$. By Shapiro's lemma, for every $H\leq G$ we have
\[
\left(\Ind_{H\cap Y^g}^H N^g\right)^H \cong \left(N^g\right)^{H\cap Y^g} \cong N^{Y\cap H^{g^{-1}}}.
\] 
Therefore (using that the content is multiplicative on direct sums, see \cite[Chap. VII, \S4, no. 5, Proposition 10]{BourbakiAC})
\[
\prod_{H \leq G} \cont{(\Ind_{Y}^G N)^H}^{n_H} = 
\prod_{H\leq G} \prod_{g\in H\backslash G/Y} \cont{N^{Y\cap H^{g^{-1}}}}^{n_H},
\]
which implies \ref{prop:rcpropres}. Assertion \ref{prop:rcpropext} follows from \cite[Remark 3.1]{Kani1994}.
\end{proof}

Mainly to simplify the treatment, we will now abandon the generality of a principal ideal domain to restrict our attention to the case $\rcring=\Z$, which will be the one of interest in the rest of the paper. We will drop the reference to the ring in the regulator constants, $\rc[\Z]{\Theta}{M}= \rc{\Theta}{M}$, to lighten notation. Note that in this case 
\[
\rc{\Theta}{M}= \prod_{H\leq G}\left( \frac{1}{\card{\tors{M}^H}^2}\det\left(\frac{1}{\card{H}}\langle\cdot,\cdot\rangle | \mt{M^H} \right)\right)^{n_H} \in \Q^\times
\] 
is a rational number.

The following proposition is basically \cite[Theorem 2.10]{Kani1994}, for which we will give a more direct proof here. Considering the case where the three modules are torsion, one immediately realizes that, for the statement to be correct, the square of the order of the torsion part is needed in the definition of the regulator constants. 

\begin{notation}
For a $\Z[G]$-module module $M$ and $i\in\mathbb{N}$, we denote by $\HH{i}{G}{M}$ the $i$-th cohomology group of $G$ with coefficients in $M$. Its order, if finite, will be denoted by $\hh{i}{G}{M}$. We will also occasionally need Tate cohomology groups $\HHT{i}{G}{M}$ (defined for all $i\in\Z$) and their orders $\hht{i}{G}{M}$.
\end{notation}

\begin{proposition}\label{prop:rces}
Let 
\[
0\to M' \stackrel{f}{\longrightarrow} M \stackrel{g}{\longrightarrow} M'' \to 0
\]
be an exact sequence of finitely generated $\Z[G]$-modules. For any Brauer relation $\Theta~=~\sum_{H\leq G} n_H H$ in $G$ we have
\[
\rc{\Theta}{M} = \rc{\Theta}{M'}\rc{\Theta}{M''}\kd{\Theta}{f}^2
\]
where
\[
\kd{\Theta}{f} =\prod_{H\leq G}\card{\Ker\left(\HH{1}{H}{f}:\HH{1}{H}{M'}\to\HH{1}{H}{M}\right)}^{n_H}.
\]
In particular, if the exact sequence splits, we have $\rc{\Theta}{M} = \rc{\Theta}{M'}\rc{\Theta}{M''}$.
\end{proposition}
\begin{proof}
We are going to use a characterization of the regulator constants introduced by Bartel and de Smit. Set
\[
P_1 = \bigoplus_{H\leq G, n_H > 0} \Z[G/H]^{n_H}, \qquad P_2 = \bigoplus_{H\leq G, n_H < 0} \Z[G/H]^{-n_H}.
\]
Since $\Theta$ is a Brauer relation in $G$, one can find an injective $\Z[G]$-homomorphism $\phi:P_1 \to P_2$ with finite cokernel. Taking duals and fixing isomorphisms $P_i \cong \dual{P_i}$ one also gets an injective $\Z[G]$-homomorphism $\phi^{\mathrm{Tr}}:P_2\to P_1$ with finite cokernel. Consider the bifunctor $\Hom_G(\cdot, \cdot)$, abbreviated $(\cdot, \cdot)$: applying $(\cdot, M)$ to $\phi$ and $\phi^{\mathrm{Tr}}$ we get $\Z[G]$-homomorphisms
\[
(\phi, M) : (P_2,M) \to (P_1,M) \quad \textrm{and} \quad (\phi^{\mathrm{Tr}}, M) : (P_1,M) \to (P_2,M).
\] 
Then by \cite[Lemma 3.1]{BarteldeSmit2013} we have
\begin{equation}\label{eq:rcformula}
\rc{\Theta}{M}=\frac{\card{\Coker(\phi^{\mathrm{Tr}}, M}/\card{\Ker(\phi^{\mathrm{Tr}}, M)}}{\card{\Coker(\phi, M)}/\card{\Ker(\phi, M)}}.
\end{equation}
Now consider the following diagram with exact rows
\[
\begin{tikzcd}
0 \arrow[r]& (P_2,M') \arrow[r]\arrow[d,"(\phi{,}M')"]& (P_2,M) \arrow[r]\arrow[d,"(\phi{,}M)"]& (P_2,M'') \arrow[r]\arrow[d,"(\phi{,}M'')"] &C(P_2,f)\arrow[r]\arrow[d,"(\phi{,}M')^1"]& 0\\
0 \arrow[r]& (P_1,M') \arrow[r]& (P_1,M) \arrow[r]& (P_1,M'') \arrow[r] &C(P_1,f)\arrow[r]& 0
\end{tikzcd}
\]
where
\[
C(P_i,f) = \Ker\left(\Ext{1}{G}{P_i}{f}:\Ext{1}{G}{P_i}{M'}\to \Ext{1}{G}{P_i}{M}\right)
\]
and $(\phi,M')^1$ is the restriction of the map $\Ext{1}{G}{P_2}{M'}\to \Ext{1}{G}{P_1}{M}$ induced by $\phi$. 
Observe that by Shapiro's lemma (see \cite[Corollary 2.8.4]{Benson1991}) 
\[
\Ext{1}{G}{P_1}{M} \cong \bigoplus_{H\leq G, n_H > 0}\Ext{1}{G}{\Z[G/H]}{M}^{n_H}\cong \bigoplus_{H\leq G, n_H > 0}\Ext{1}{H}{\Z}{M}^{n_H} = \bigoplus_{H\leq G, n_H > 0}\HH{1}{H}{M}^{n_H}
\]
and similarly for $P_2$.
In particular, since these groups are finite, we have
\[
\frac{\card{\Coker(\phi, M')^1}}{\card{\Ker(\phi, M')^1}} = \frac{\card{C(P_1,f)}}{\card{C(P_2,f)}} = \kd{\Theta}{f}.
\]
Splitting the above diagram in two diagrams, each having 3-term short exact sequences as rows, and applying the snake lemma to each of them, we deduce that 
\[
\frac{\card{\Coker(\phi, M)}}{\card{\Ker(\phi, M)}} = \frac{\card{\Coker(\phi, M')}}{\card{\Ker(\phi, M')}}\frac{\card{\Coker(\phi, M'')}}{\card{\Ker(\phi, M'')}}\kd{\Theta}{f}^{-1}.
\] 
A similar argument applied to $(\phi^\mathrm{Tr},M)$ implies
\[
\frac{\card{\Coker(\phi^\mathrm{Tr}, M)}}{\card{\Ker(\phi^\mathrm{Tr}, M)}} = \frac{\card{\Coker(\phi^\mathrm{Tr}, M')}}{\card{\Ker(\phi^\mathrm{Tr}, M')}}\frac{\card{\Coker(\phi^\mathrm{Tr}, M'')}}{\card{\Ker(\phi^\mathrm{Tr}, M'')}}\kd{\Theta}{f}
\] 
concluding the proof.
\end{proof}

\begin{remark}
Kani's defect $\kd{\Theta}{f}$ for a specific $f$ appears in \cite{Bartel2012}, where the exact sequence under consideration is the one involving $S$-units of a number fields and their torsion
\[
0\to \roots{K}\to\units{K}{S}\to \mt{\units{K}{S}}\to 0
\]
and $\kd{\Theta}{f}$ is called $\lambda$. Note also that \eqref{eq:rcformula}, the main ingredient of our proof, is analogous to the recipe to compute Kani's discriminant given in \cite[Remark 5.17]{Kani1994}. 
\end{remark}

The statement of the following proposition is the same as \cite[Theorem 2.9]{Kani1994} but we give here a more direct proof. Again our modified definition of regulator constants is necessary for the statement to be correct in general. 

\begin{proposition}\label{prop:rcct}
Let $M$ be a finitely generated cohomologically trivial $\Z[G]$-module. Then for any Brauer relation $\Theta$ in $G$ we have $\rc{\Theta}{M}=1$.
\end{proposition}
\begin{proof}
We first prove the assertion when $M$ is a finitely generated projective $\Z[G]$-module. For a prime $p$, let $\Z_{(p)}$ denote the localization of $\Z$ at $(p)$. Then $M\otimes_\Z \Z_{(p)}$ is $\Z_{(p)}[G]$-free for every prime $p$ by a theorem of Swan (see \cite{Swan1960}, \cite[Theorem 32.11]{CurtisReiner1981}). In particular $\rc{\Theta}{M\otimes_\Z \Z_{(p)}}$ is trivial. By Proposition \ref{prop:rcprop} we deduce that the $p$-adic valuation of $\rc{\Theta}{M}$ is $0$ for every prime $p$, which means that $\rc{\Theta}{M}=1$.

Let now $M$ be a finitely generated cohomologically trivial $\Z[G]$-module. Then there is an exact sequence of finitely generated $\Z[G]$-modules
\[
0 \to Q_2 \to Q_1 \to M \to 0
\]   
where the $Q_i$ are projective (see \cite[Chapter IX, Theorem 8]{Serre2013}). We then conclude by Proposition \ref{prop:rces}.
\end{proof}

Combining Propositions \ref{prop:rces} and \ref{prop:rcct}, one immediately gets the following corollary (this is to be compared with \cite[Corollary 2.11]{Kani1994}, where, however, the exponent in the definition of $\chi$ needs to be corrected). 
\begin{corollary}\label{cor:rces}
Let 
\[
0 \to M' \to M_r \to \cdots \to M_1 \to M'' \to 0
\]
be an exact sequence of finitely generated $\Z[G]$-modules. If the $M_i$ are cohomologically trivial, then for any Brauer relation $\Theta=\sum_{H\leq G} n_H H$ we have
\[
\rc{\Theta}{M'} = \rc{\Theta}{M''}^{(-1)^r} \prod_{H\leq G}\left(\prod_{i=1}^r \hh{i}{H}{M'}^{(-1)^{i}}\right)^{2n_H}.
\]
\end{corollary}

\section{Tate exact sequences for units}\label{sec:tes}
In this section we briefly recall the two four-term exact sequences introduced by Ritter and Weiss in \cite{RitterWeiss1996} which are the starting point of our proof of \eqref{eq:bcnf}. 

We begin with the definition of the inertial lattices in a local setting. Let $p$ be a rational prime and let $D$ be the Galois group of a finite extension $F/E$ of $p$-adic fields, whose corresponding residual fields extension has Galois group $\overline{D}$, generated by the Frobenius $\varphi$. The units of $F$ will be denoted by $\locunits{F}$.  

\begin{definition}
The inertial lattice of $F/E$ is the $\Z[D]$-lattice
\[
W = \{(x,y)\in\Delta D\oplus \Z[\overline{D}] : \overline x = (\varphi-1)y  \}
\]
where the bar denotes the projection $D\to \overline{D}$ and $\Delta D=\Ker(\Z[D]\to \Z)$ is the augmentation ideal.
\end{definition}

The following properties of inertial lattices will be used to compute their regulator constants.

\begin{proposition}[{\cite[Lemma 5, Proposition 2]{RitterWeiss1996}, \cite[Chapter 7]{Weiss1996}}]\label{prop:il}
There is a commutative diagram of $\Z[D]$-modules with exact rows and columns :
\begin{equation}\label{eq:wdiag}
\begin{tikzcd}
&0 \arrow[d]& 0 \arrow[d] \\
&\locunits{F} \arrow[d]\ar[equal]{r}& \locunits{F} \arrow[d] \\
0\arrow[r]&F^\times \arrow[d,"v"]\arrow[r] & V\arrow[d]\arrow[r]& \Delta D \ar[equal]{d}\arrow[r]& 0\\
0\arrow[r]&\Z \arrow[r]\arrow[d] & W\arrow[r]\arrow[d]& \Delta D \arrow[r]& 0\\
&0 & 0 \\
\end{tikzcd}
\end{equation}
where $v$ is the valuation of $F$ and $V$ is cohomologically trivial. Moreover there is an exact sequence of $\Z[D]$-modules
\begin{equation}\label{eq:wdualwes}
0\to W \to \Z[D]\oplus \Z[D] \to \dual{W}\to 0.
\end{equation}
\end{proposition}

\begin{remark}
To avoid repeated citation, we recall here the main cohomological local class field theory result we are going to use, \ie the isomorphism
\[
 \HHT{i}{D}{\Z} \cong \HHT{i+2}{D}{F^\times} 
\] 
(see \cite[Corollary 7.1.7]{NSW2013}). This is of course compatible with the middle horizontal line of \eqref{eq:wdiag}. We will also use the fact that $U_F$ is cohomologically trivial is $F/E$ is unramified (see \cite[Proposition 7.1.2]{NSW2013}).
\end{remark}

\begin{lemma}\label{lemma:rcwloc}
For any Brauer relation $\Theta=\sum_{H\leq D} n_H H$ in $D$ we have $\rc{\Theta}{\dual{W}} = 1$.
\end{lemma}
\begin{proof}
From the exact sequence \eqref{eq:wdualwes} and Proposition \ref{prop:rces} we deduce
\[
\rc{\Theta}{\dual{W}} = \rc{\Theta}{W}^{-1}\prod_{H\leq D}\hh{1}{H}{W}^{-2n_H}.
\]
Applying again Proposition \ref{prop:rces} to the bottom row of diagram \eqref{eq:wdiag} we can rewrite this as
\begin{align*}
\rc{\Theta}{\dual{W}} &= \rc{\Theta}{\Z}^{-1}\rc{\Theta}{\Delta D}^{-1}\prod_{H\leq D}\hh{1}{H}{W}^{-2n_H}\\
&= \prod_{H\leq D}\left(\frac{\hh{1}{H}{\Delta D}}{\hh{1}{H}{W}}\right)^{2n_H}
\end{align*}
where the last equality follows by applying once more Proposition \ref{prop:rces} to the augmentation exact sequence $0\to \Delta D \to \Z[D] \to \Z\to0$. Now since $V$ in Diagram \eqref{eq:wdiag} is cohomologically trivial, by dimension shifting we get 
\begin{align*}
\frac{\hh{1}{H}{\Delta D}}{\hh{1}{H}{W}} &= \frac{\hh{2}{H}{F^\times}}{\hh{2}{H}{\locunits{F}}} \\
&= \frac{\hh{2}{H}{F^\times}}{\hh{1}{H}{\locunits{F}}}
\end{align*}
where the last equality follows from \cite[Corollary 4.3]{GruenbergWeiss1996}.
Now we have $\hh{2}{H}{F^\times} = \hht{0}{H}{\Z}=\card{H}$ by local class field theory. Moreover $\hh{1}{H}{\locunits{F}}$ is equal to the ramification index of $F/F^{H}$, as it follows easily by taking $H$-invariants of the valuation sequence $0\to \locunits{F}\to F^\times \to \Z \to 0$ and using $\HH{1}{H}{F}=0$ by Hilbert's Theorem 90. We have therefore obtained 
\[
\rc{\Theta}{\dual{W}} = \prod_{H\leq D} [H : H \cap I]^{2n_H}
\]
where $I\leq D$ is the inertia subgroup of $F/E$. Observe now that
\[
\prod_{H\leq D} [D : I]^{n_H} = 1
\]
(see \cite[Example 2.30]{DokchitserDokchitser2009}), which allows us to write
\[
\rc{\Theta}{\dual{W}} = \prod_{H\leq D} \left(\frac{[D : I]}{[H : H \cap I]}\right)^{-2n_H}.
\]
The right-hand side of the above equality is trivial by \cite[Theorem 2.36 (f)]{DokchitserDokchitser2009}, showing the claim of the lemma. 
\end{proof}


We now leave the local setting and consider a finite Galois extension $K/k$ of number field with Galois group $G$. We will be working with a finite $G$-invariant set $S$ of primes of $K$, containing the archimedean primes. 

\begin{notation}
For a subextension $F/k$ of $K/k$, let $S_F$ be the set of primes of $F$ which lie below primes of $S$. We will denote by $\units{F}{S}$ the group of $S_F$-units of $F$ and set $\roots{F}=\tors{(\units{F}{S})}$ for the group of roots of unity. The $S_F$-class group of $F$ will be denoted by $\Cl{F}{S}$.
\end{notation}

\begin{notation}
For a prime $\mathfrak{P}$ of $K$, we will denote by $K_\mathfrak{P}$ the completion of $K$ at $\mathfrak{P}$. The group of units $\locunits{K_\mathfrak{P}}$ of $K_\mathfrak{P}$ will be denoted simply by $\locunits{\mathfrak{P}}$ to avoid double subscripts, if no confusion may arise. Similarly, the inertial sublattice of $K_\mathfrak{P}$ will be denoted by $W_\mathfrak{P}$. For a subgroup $H$ of $G$, $H_\mathfrak{P}$ denotes the decomposition subgroup of $\mathfrak{P}$ in $H$, or equivalently, the Galois group of $K_\mathfrak{P}$ over $(K^H)_\mathfrak{p}$ where $\mathfrak{p} = \mathfrak{P} \cap K^H$. 
\end{notation}

Let $*$ denote a choice of representatives for each orbit of the action of $G$ on the primes of $K$ and let $S_*$ denote the intersection of $S$ and $*$. Let $S_*^{ram}$ denote the set of ramified primes in $*$ which are \emph{not} in $S$.  

\begin{theorem}[{\cite{RitterWeiss1996}, \cite[Theorem 9]{Weiss1996}}]\label{thm:rwes}
There is an exact sequence of finitely generated $\Z[G]$-modules
\begin{equation}\label{eq:mainrw}
0 \to \units{K}{S} \to A \to B \to \nabla \to 0 
\end{equation}
where $A$ and $B$ are cohomologically trivial and $\nabla$ fits into a short exact sequence of $\Z[G]$-modules
\begin{equation}\label{eq:pies}
0\to \Cl{K}{S}\stackrel{\iota}{\longrightarrow} \nabla \stackrel{\pi}{\longrightarrow} \overline{\nabla}\to 0
\end{equation}
where $\Cl{K}{S}$ is sent isomorphically onto $\tors{\nabla}$. On its turn, $\overline{\nabla}$ fits into an exact sequence of $\Z[G]$-modules
\begin{equation}\label{eq:epsilones}
0\to \overline{\nabla}\stackrel{\eta}{\longrightarrow} \Z[S] \oplus \WW{S}\stackrel{\varepsilon}{\longrightarrow} \Z \to 0
\end{equation}
where 
\[
\WW{S} = \bigoplus_{\mathfrak{P}\in S_*^{ram}} \Ind^G_{G_{\mathfrak{P}}} \dual{(W_{\mathfrak{P}})}.
\]
\end{theorem} 

\section{Main result}
We now come to the proof of Brauer's class number formula. We will first focus on establishing a relation between the regulator constants of the units and a ratio of class numbers, starting from the exact sequences of the previous section and using the formalism of Section \ref{sec:rc} and (cohomological) class field theory. We will then recall the  known relation between regulator constants of the units and regulators (see Example \ref{ex:rcexa}), which we use to conclude the proof of \eqref{eq:bcnf}. 

We will be working in the context of the introduction and the final part of the previous section and retain the notation introduced so far: $K/k$ is a finite Galois extension of number fields with Galois group $G$ and $S$ is a finite $G$-invariant set of primes of $K$ containing the archimedean primes. 

\begin{notation}
Let $F/k$ be a subextension of $K/k$. Then $\idCl{F}$ denotes the id\`eles class group of $F$. Let 
\[
\idunits{F}{S}= \prod_{\mathfrak{p}\in S_F} (F_\mathfrak{p})^\times \times \prod_{\mathfrak{p}\notin S_F} U_\mathfrak{p}
\] 
denote the group of $S_F$-idelic units of $F$ and let $\Qq{F}{S}$ be the quotient of the diagonal embedding $\units{F}{S}\to \idunits{F}{S}$. 
\end{notation}

The reason for introducing $\Qq{K}{S}$ is to split the four terms exact sequence
\begin{equation}\label{eq:euccl}
0\to \units{K}{S}\to \idunits{K}{S}\to\idCl{K}\to \Cl{K}{S}\to 0
\end{equation}
in two short exact sequences
\begin{equation}\label{eq:euq}
0\to \units{K}{S}\to \idunits{K}{S}\to \Qq{K}{S}\to 0,
\end{equation}
\begin{equation}\label{eq:qccl}
0\to \Qq{K}{S}\to \idCl{K}\to \Cl{K}{S}\to 0.
\end{equation}

We start by computing the $\psi$ factors of Proposition \ref{prop:rces} for the homomorphisms $\iota$ and $\eta$ appearing in Theorem \ref{thm:rwes}. The main technical ingredient is the next lemma.


\begin{lemma}\label{lemma:keylemma}
For any $i\in\Z$ and any $H\leq G$, there is an anticommutative diagram
\[
\begin{tikzcd}
\HH{i}{H}{\Cl{K}{S}} \arrow[d]\arrow[r]& \HH{i}{H}{\nabla} \arrow[d] \\
\HH{i+1}{H}{\Qq{K}{S}} \arrow[r] & \HH{i+2}{H}{\units{K}{S}}
\end{tikzcd}
\]
where the left-hand vertical map comes from \eqref{eq:qccl}, the bottom horizontal map comes from \eqref{eq:euq}, the top horizontal map comes from \eqref{eq:pies} and the right-hand vertical map is the isomorphism induced by a Tate exact sequence with the same properties as \eqref{eq:mainrw}.
\end{lemma}
\begin{proof}
We are going to use the notation of \cite{RitterWeiss1996} with the exception of the $S$-unit id\`eles which are denoted $J_{S}$ in \cite{RitterWeiss1996} and $\idunits{K}{S}$ here. Also, we will keep the reference to the set $S$ here, while in \cite{RitterWeiss1996} it is dropped. We will avoid recalling definitions and notation which are not explicitly needed for the proof, referring the reader to \cite{RitterWeiss1996} for more details.

By \cite[Theorem 1]{RitterWeiss1996}, there is a commutative diagram of $\Z[G]$-modules 
\[
\begin{tikzcd}
0 \arrow[r]& \idunits{K}{S} \arrow[d] \arrow[r]& V_{S'} \arrow[d,"\theta"] \arrow[r]& W_{S'} \arrow[d] \arrow[r]& 0\\ 
0 \arrow[r]& \idCl{K} \arrow[r]& \mathfrak{V}\arrow[r]& \Delta G \arrow[r]& 0
\end{tikzcd}
\] 
where $\theta$ is surjective and the leftmost vertical map is the one in \eqref{eq:euccl}. Applying the snake lemma to the above diagram and using \eqref{eq:euccl}, one obtains a ``Tate" exact sequence
\begin{equation}\label{eq:falseTatees}
0\to \units{K}{S} \to A_{\theta} \to R_{S'} \to \Cl{K}{S}\to 0 
\end{equation}
where $A_{\theta}$ is the kernel of the central vertical map and $R_{S'}$ the one of the rightmost vertical map. 
Let $\tilde{R}_{S'}$ denote the inverse image of $R_{S'}$ under the map $V_{S'} \to W_{S'}$, so that we have a commutative diagram with exact rows and injective vertical maps as follows
\begin{equation}\label{eq:diagrw1}
\begin{tikzcd}
0 \arrow[r]& \units{K}{S} \arrow[d] \arrow[r]& A_{\theta} \arrow[d] \arrow[r]& A'_{\theta} \arrow[d] \arrow[r]& 0\\ 
0 \arrow[r]& \idunits{K}{S} \arrow[r]& \tilde{R}_{S'} \arrow[r]& R_{S'} \arrow[r]& 0
\end{tikzcd}
\end{equation}
where $A'_{\theta} = \Ker(R_{S'} \to \Cl{K}{S})$. Observe also that, since $\theta$ is surjective, the image of $\tilde{R}_{S'}$ under the map $V_{S'}\to \mathfrak{V}$ equals the image of $\idCl{K}$ under the map $\idCl{K}\to \mathfrak{V}$. In particular we can complete the above diagram to obtain a $3\times 3$ commutative diagram 
\[
\begin{tikzcd}
&0 \arrow[d]& 0\arrow[d]& 0\arrow[d] \\
0 \arrow[r]&\units{K}{S} \arrow[d]\arrow[r]&\idunits{K}{S} \arrow[d]\arrow[r]&\Qq{K}{S} \arrow[d]\arrow[r]& 0 \\
0 \arrow[r]&A_{\theta} \arrow[d]\arrow[r]&\tilde{R}_{S'} \arrow[d]\arrow[r]&\idCl{K} \arrow[d]\arrow[r]& 0 \\
0 \arrow[r]&A'_{\theta} \arrow[d]\arrow[r]&R_{S'} \arrow[d]\arrow[r]&\Cl{K}{S} \arrow[d]\arrow[r]& 0 \\
&0 & 0& 0
\end{tikzcd}
\]
where the top horizontal sequence is \eqref{eq:euq}, the rightmost vertical sequence is \eqref{eq:qccl} and the leftmost vertical sequence is the top sequence of diagram \eqref{eq:diagrw1}. 
Applying \cite[Chapter III, Proposition 4.1]{CartanEilenberg1956}, we get an anticommutative square
\begin{equation}\label{eq:anticomm}
\begin{tikzcd}
\HH{i}{H}{\Cl{K}{S}} \arrow[d]\arrow[r]& \HH{i+1}{H}{A'_\theta} \arrow[d] \\
\HH{i+1}{H}{\Qq{K}{S}} \arrow[r] & \HH{i+2}{H}{\units{K}{S}}
\end{tikzcd}
\end{equation}
The exact sequence \eqref{eq:falseTatees} can be transformed into \eqref{eq:mainrw} using a commutative diagram of $\Z[G]$-modules with exact rows (see \cite[p. 163]{RitterWeiss1996})
\[
\begin{tikzcd}
0 \arrow[r]& R_{S'} \arrow[d] \arrow[r]& B_{S'} \arrow[d] \arrow[r]& \overline{\nabla}_{*} \ar[equal]{d} \arrow[r]& 0\\ 
0 \arrow[r]& \Cl{K}{S} \arrow[r]& \nabla_{\theta} \arrow[r]& \overline{\nabla}_{*} \arrow[r]& 0
\end{tikzcd}
\]
with $B_{S'}$ projective. In particular $A'_{\theta}$ is isomorphic to $B'_{S'} = \Ker(B_{S'}\to \nabla_{\theta})$ and we can extend the top exact sequence of diagram \eqref{eq:diagrw1} to a four-terms exact sequence 
\[
0 \to \units{K}{S} \to A_\theta \to B'_{S'} \to \nabla_\theta \to 0
\]
which has the same properties as the one in Theorem \ref{thm:rwes}. Since $B_{S'}$ is projective we deduce 
\[ 
\HH{i+1}{H}{A'_\theta}\cong \HH{i+1}{H}{B'_{S'}}\cong \HH{i}{H}{\nabla_{\theta}}. 
\]
Plugging this in Diagram \eqref{eq:anticomm} we complete the proof.
\end{proof}


\begin{notation}
In what follows we let 
\[
\delta^i_{\Cl{}{}}:\HH{i}{H}{\Cl{K}{S}}\to \HH{i+1}{H}{\Qq{K}{S}}
\] 
denote the connecting homomorphism induced by \eqref{eq:qccl} and 
\[
\delta^i_{\Qq{}{}}: \HH{i}{H}{\Qq{K}{S}}\to \HH{i+1}{H}{\units{K}{S}}
\]
denote the connecting homomorphism induced by \eqref{eq:euq}. Moreover we set $\delta^i = \delta^{i+1}_{\Qq{}{}}\circ \delta^{i}_{\Cl{}{}}$ so that we have a commutative triangle
\[
\begin{tikzcd}
\HH{i}{H}{\Cl{K}{S}} \arrow[d,"\delta^{i}_{\Cl{}{}}"]\arrow[rd,"\delta^i"]& \\
\HH{i+1}{H}{\Qq{K}{S}} \arrow[r,"\delta^{i+1}_{\Qq{}{}}"] & \HH{i+2}{H}{\units{K}{S}}
\end{tikzcd}
\]
\end{notation}

\begin{remark}
As for the local setting, we recall once and for all the main results of cohomological global class field theory we are going to use, \ie isomorphisms
\[
\HHT{i+2}{G}{\idCl{K}} \cong \HHT{i}{G}{\Z}
\]
(see \cite[Theorem 8.1.22]{NSW2013}).
\end{remark}

The next lemma gives a formula for 
\begin{equation}\label{eq:psiiota}
\kd{\Theta}{\iota} = \prod_{H\leq G} \card{\Ker(\HH{1}{H}{\Cl{K}{S}}\to \HH{1}{H}{\nabla})}^{n_H}
\end{equation}
where $\Theta= \sum_{H\leq G} n_H H$ is a Brauer relation in $G$.

\begin{lemma}\label{lemma:psiiota}
We have
\[
\card{\Ker(\HH{1}{H}{\Cl{K}{S}}\to \HH{1}{H}{\nabla})} = \frac{\hh{2}{H}{\Qq{K}{S}}}{\hh{3}{H}{\units{K}{S}}}\frac{\card{\Coker \delta^1}}{\card{\Coker\delta^{1}_{\Cl{}{}}}}.
\]
\end{lemma}
\begin{proof}
In the diagram 
\[
\begin{tikzcd}
\HH{1}{H}{\Cl{K}{S}} \arrow[d,"\delta^{1}_{\Cl{}{}}"]\arrow[rd,"\delta^1"]\arrow[r]& \HH{1}{H}{\nabla} \arrow[d] \\
\HH{2}{H}{\Qq{K}{S}} \arrow[r,"\delta^{2}_{\Qq{}{}}"] & \HH{3}{H}{\units{K}{S}}
\end{tikzcd}
\]
the bottom left triangle is commutative and the top right one is anticommutative by Lemma \ref{lemma:keylemma}. Since $\HH{1}{H}{\idCl{K}}=0$, we have $\Ker \delta^{1}_{\Cl{}{}} =0$ and an exact sequence
\[
0 \to \Ker \delta^1 \to \Ker \delta^{2}_{\Qq{}{}} \to \Coker \delta^{1}_{\Cl{}{}} \to \Coker\delta^1 \to \Coker \delta^{2}_{\Qq{}{}} \to 0.
\]
Since the right vertical map of the above diagram is an isomorphism, we deduce that
\[
\card{\Ker(\HH{1}{H}{\Cl{K}{S}}\to \HH{1}{H}{\nabla})} = \card{\Ker \delta^1} = \frac{\card{\Ker \delta^{2}_{\Qq{}{}}}}{\card{\Coker \delta^{2}_{\Qq{}{}}}}\frac{\card{\Coker \delta^1\vphantom{\Ker \delta^{2}_{\Qq{}{}}}}}{\card{\Coker \delta^{1}_{\Cl{}{}}}}=\frac{\hh{2}{H}{\Qq{K}{S}}}{\hh{3}{H}{\units{K}{S}}}\frac{\card{\Coker \delta^1}}{\card{\Coker \delta^{1}_{\Cl{}{}}}}
\]
which is the statement of the lemma.
\end{proof}

Observe that since $\HH{1}{H}{\Z}=\HH{1}{H}{\Z[S]}=0$, for every Brauer relation $\Theta = \sum_{H \leq G} n_H H$ we have
\begin{equation}\label{eq:psieta}
\kd{\Theta}{\eta} = \prod_{H\leq G} \left(\frac{\hh{1}{H}{\overline{\nabla}}}{\hh{1}{H}{\WW{S}}}\right)^{n_H}
\end{equation}
for the map $\eta$ appearing in \eqref{eq:epsilones}.

\begin{lemma}\label{lemma:psieta}
We have
\[
\hh{1}{H}{\overline{\nabla}} = \hh{2}{H}{\idCl{K}}\frac{\card{\Coker \delta^1}}{\card{\Coker\delta^{1}_{\Cl{}{}}}}\card{\Ker \delta^3_{\Qq{}{}}}.
\]
\end{lemma}
\begin{proof}
From the exact sequence
\[
0\to \Coker\left(\HH{1}{H}{\Cl{K}{S}}\to \HH{1}{H}{\nabla}\right) \to \HH{1}{H}{\overline{\nabla}}\to \Ker\left(\HH{2}{H}{\Cl{K}{S}}\to\HH{2}{H}{\nabla}\right)\to 0
\]
we deduce
\[
\hh{1}{H}{\overline{\nabla}} = \card{\Ker(\HH{2}{H}{\Cl{K}{S}}\to\HH{2}{H}{\nabla})}\cdot\card{\Coker(\HH{1}{H}{\Cl{K}{S}}\to \HH{1}{H}{\nabla})}.
\]
Now in the diagram
\[
\begin{tikzcd}
\HH{2}{H}{\Cl{K}{S}} \arrow[d,"\delta^2_{\Cl{}{}}"]\arrow[rd,"\delta^2"]\arrow[r]& \HH{2}{H}{\nabla} \arrow[d] \\
\HH{3}{H}{\Qq{K}{S}} \arrow[r,"\delta^3_{\Qq{}{}}"] & \HH{4}{H}{\units{K}{S}}
\end{tikzcd}
\]
the bottom left triangle is commutative and the top left one is anticommutative by Lemma \ref{lemma:keylemma}. We deduce an exact sequence
\[
0 \to \Ker \delta^2_{\Cl{}{}} \to \Ker \delta^2 \to \Ker \delta^3_{\Qq{}{}} \to 0
\]
where exactness on the right comes from 
\[
\Coker \delta^2_{\Cl{}{}}  \subseteq  \HH{3}{H}{\idCl{K}} \cong \HH{1}{H}{\Z} = 0
\]
by class field theory. Finally observe that from the long exact cohomology sequence of \eqref{eq:qccl} we can extract a short exact sequence
\[
0 \to \Coker \delta^1_{\Cl{}{}} \to \HH{2}{H}{\idCl{K}}\to \Ker \delta^2_{\Cl{}{}}\to 0.
\]
Putting everything together and using that the right-hand vertical map of the diagram of Lemma \ref{lemma:keylemma} is an isomorphism for any $i$, we deduce
\begin{align*}
\hh{1}{H}{\overline{\nabla}} &= \card{\Ker(\HH{2}{H}{\Cl{K}{S}}\to\HH{2}{H}{\nabla})}\cdot\card{\Coker(\HH{1}{H}{\Cl{K}{S}}\to \HH{1}{H}{\nabla})}\\
&=\card{\Ker \delta^2}\cdot \card{\Coker \delta^1}\\
&=\card{\Ker \delta^2_{\Cl{}{}}}\cdot\card{\Ker \delta^3_{\Qq{}{}}}\cdot\card{\Coker \delta^1}\\
&=\frac{\hh{2}{H}{\idCl{K}}}{\card{\Coker \delta^1_{\Cl{}{}}}}\cdot\card{\Ker \delta^3_{\Qq{}{}}}\cdot\card{\Coker \delta^1}.\qedhere
\end{align*}
\end{proof}

\begin{lemma}\label{lemma:h1wh3u}
For any subgroup $H \leq G$ we have
\[
\hh{1}{H}{\WW{S}} = \hh{3}{H}{\idunits{K}{S}}.
\]
\end{lemma}
\begin{proof}
Let $\mathfrak{P}$ be a prime of $K$. If $\mathfrak{P}$ is unramified in $K/k$, then $\hh{3}{H_{\mathfrak{P}}}{\locunits{\mathfrak{P}}}=1$. If, instead, $\mathfrak{P}$ is ramified in $K/k$ but belongs to $S$, then the local component of $\HH{3}{H}{\idunits{K}{S}}$ at $\mathfrak{P}$ is $\HH{3}{H_{\mathfrak{P}}}{K^\times_{\mathfrak{P}}}$. This is isomorphic to $\HH{1}{H_{\mathfrak{P}}}{\Z}$ by local class field theory and hence trivial. So, by Shapiro's lemma, we are left to prove that
\[
\hh{1}{H_{\mathfrak{P}}}{\dual{W_{\mathfrak{P}}}} = \hh{3}{H_{\mathfrak{P}}}{\locunits{\mathfrak{P}}}
\]
for $\mathfrak{P}$ ramified and not in $S$. By Proposition \ref{prop:il}, there is an exact sequence of $H_{\mathfrak{P}}$-modules 
\[
0\to \locunits{\mathfrak{P}} \to V_\mathfrak{P} \to W_\mathfrak{P} \to 0
\]
where $V_\mathfrak{P}$ is cohomologically trivial. We deduce that
\[
\HH{3}{H_{\mathfrak{P}}}{\locunits{\mathfrak{P}}} \cong \HH{2}{H_{\mathfrak{P}}}{W_{\mathfrak{P}}}.
\]
Moreover, by \eqref{eq:wdualwes}, we have an isomorphism
\[
\HH{2}{H_{\mathfrak{P}}}{W_{\mathfrak{P}}}\cong\HH{1}{H_{\mathfrak{P}}}{\dual{(W_{\mathfrak{P}})}}
\]
which concludes the proof.
\end{proof}

For future use we collect the above results in the following corollary.
\begin{corollary}\label{cor:psiratio}
For a Brauer relation $\Theta= \sum_{H\leq G} n_H H$ in $G$ we have
\[
\frac{\kd{\Theta}{\iota}}{\kd{\Theta}{\eta}} = \prod_{H\leq G} \left(\frac{\card{\Coker\left(\HH{2}{H}{\units{K}{S}}\to\HH{2}{H}{\idunits{K}{S}}\right)}}{\hh{2}{H}{\idCl{K}}}\right)^{n_H}
\]
where $\iota$ and $\eta$ are the maps appearing in \eqref{eq:pies} and \eqref{eq:epsilones}, respectively.
\end{corollary}
\begin{proof}
We have
\begin{align*}
\frac{\kd{\Theta}{\iota}}{\kd{\Theta}{\eta}} &= \prod_{H\leq G} \left(\frac{\hh{2}{H}{\Qq{K}{S}}\hh{1}{H}{\WW{S}}}{\hh{3}{H}{\units{K}{S}}\hh{2}{H}{\idCl{K}}\card{\Ker\delta^3_{\Qq{}{}}}}\right)^{n_H}\quad \text{(\eqref{eq:psiiota}, \eqref{eq:psieta}, Lemmas \ref{lemma:psiiota} and \ref{lemma:psieta})}\\
&= \prod_{H\leq G} \left(\frac{\hh{2}{H}{\Qq{K}{S}}\hh{3}{H}{\idunits{K}{S}}}{\hh{3}{H}{\units{K}{S}}\hh{2}{H}{\idCl{K}}\card{\Ker\delta^3_{\Qq{}{}}}}\right)^{n_H}\quad \text{(Lemma \ref{lemma:h1wh3u})}\\
&=\prod_{H\leq G} \left(\frac{\card{\Coker\left(\HH{2}{H}{\units{K}{S}}\to\HH{2}{H}{\idunits{K}{S}}\right)}}{\hh{2}{H}{\idCl{K}}}\right)^{n_H}
\end{align*}
where the last equality comes from the long exact cohomology sequence of \eqref{eq:euq}.
\end{proof}

The term involving $\WW{S}$ in the above corollary can be rewritten in terms of idelic units, as the following lemma shows.

We now apply the results of Section \ref{sec:rc} to the exact sequences of Theorem \ref{thm:rwes} to get a first formula for the regulator constants of the units. The formula will be further simplified later.

\begin{proposition}\label{prop:firststep}
For any Brauer relation $\Theta=\sum_{H\leq G} n_H H$ in $G$ we have
\[
\rc{\Theta}{\units{K}{S}} =\rc{\Theta}{\Cl{K}{S}} \rc{\Theta}{\Z}^{-1} \rc{\Theta}{\Z[S]} \rc{\Theta}{\WW{S}} \left(\frac{\kd{\Theta}{\iota}}{\kd{\Theta}{\eta}}\right)^2 \prod_{H\leq G}\left(\frac{\hh{2}{H}{\units{K}{S}}}{\hh{1}{H}{\units{K}{S}}}\right)^{2n_H}.
\]
\end{proposition}
\begin{proof}
Starting with the exact sequence \eqref{eq:mainrw} and using Corollary \ref{cor:rces} and Proposition \ref{prop:rces} we get
\begin{align*}
\rc{\Theta}{\units{K}{S}} & = \rc{\Theta}{\nabla} \prod_{H\leq G}\left(\frac{\hh{2}{H}{\units{K}{S}}}{\hh{1}{H}{\units{K}{S}}}\right)^{2n_H} \\
& = \rc{\Theta}{\Cl{K}{S}} \rc{\Theta}{\overline{\nabla}} \kd{\Theta}{\iota}^2 \prod_{H\leq G}\left(\frac{\hh{2}{H}{\units{K}{S}}}{\hh{1}{H}{\units{K}{S}}}\right)^{2n_H}\\
& = \rc{\Theta}{\Cl{K}{S}} \rc{\Theta}{\Z}^{-1} \rc{\Theta}{\Z[S]} \rc{\Theta}{\WW{S}} \left(\frac{\kd{\Theta}{\iota}}{\kd{\Theta}{\eta}}\right)^2 \prod_{H\leq G}\left(\frac{\hh{2}{H}{\units{K}{S}}}{\hh{1}{H}{\units{K}{S}}}\right)^{2n_H}. \qedhere
\end{align*}
\end{proof}

We now proceed to compute the regulator constants appearing in the right-hand side of the formula of Proposition \ref{prop:firststep}.

\begin{lemma}\label{lemma:rcws}
For any Brauer relation $\Theta$ in $G$ we have $\rc{\Theta}{\WW{S}} = 1$.
\end{lemma}
\begin{proof}
By Propositions \ref{prop:rces} and \ref{prop:rcprop} and Lemma \ref{lemma:rcwloc} we have 
\[
\rc{\Theta}{\WW{S}} = \prod_{\mathfrak{P}\in S_*^{ram}} \rc{\Theta}{\Ind^G_{G_{\mathfrak{P}}} \dual{(W_{\mathfrak{P}})}} = \prod_{\mathfrak{P}\in S_*^{ram}}\rc{\Res^G_{G_{\mathfrak{P}}}\Theta}{\dual{(W_{\mathfrak{P}})}} = 1.\qedhere
\]
\end{proof}

\begin{lemma}\label{lemma:rczs}
For any Brauer relation $\Theta=\sum_{H\leq G}n_H H$ in $G$ we have 
\[
\rc{\Theta}{\Z[S]} = \prod_{H\leq G}\left(\frac{\hh{1}{H}{\idunits{K}{S}}}{\hh{2}{H}{\idunits{K}{S}}}\right)^{n_H}.
\]
\end{lemma}
\begin{proof}
For a prime $\mathfrak{P}$ of $K$, let $M_{\mathfrak{P}}$ denote $K_{\mathfrak{P}}^\times$ if $\mathfrak{P}\in S$ and $\locunits{\mathfrak{P}}$ otherwise. Then, by Shapiro's lemma, one has for any $i$
\[
\HH{i}{H}{\idunits{K}{S}} = \prod_{\mathscr{p}}\prod_{\mathfrak{p}|\mathscr{p}}
\HH{i}{H}{\oplus_{\mathfrak{P}|\mathfrak{p}}M_{\mathfrak{P}}} = \prod_{\mathscr{p}}\prod_{\mathfrak{p}|\mathscr{p}} \HH{i}{H_{\tilde{\mathfrak{P}}}}{M_{\tilde{\mathfrak{P}}}}
\]
where the outer product is taken over all primes $\mathscr{p}$ in $k$, the inner product over primes $\mathfrak{p}$ in $K^H$ (dividing a given $\mathscr{p}$ in $k$) and the sum in the cohomology group is over primes $\mathfrak{P}$ of $K$ (dividing a given $\mathfrak{p}$ in $K^H$). Moreover, for any prime $\mathfrak{p}$ of $K^H$, we have fixed a prime $\tilde{\mathfrak{P}}$ of $K$ dividing $\mathfrak{p}$. 
 
For a prime  $\mathfrak{P}$ of $K$ not in $S$, we have $\hh{1}{H_{\mathfrak{P}}}{M_{\mathfrak{P}}}=\hh{2}{H_{\mathfrak{P}}}{M_{\mathfrak{P}}}$ by \cite[Corollary 4.3]{GruenbergWeiss1996}. If, instead, $\mathfrak{P}$ belongs to $S$, we have $\HH{1}{H_{\mathfrak{P}}}{M_{\mathfrak{P}}}=0$ by Hilbert's Theorem 90. Moreover by local class field theory $\HH{2}{H_{\mathfrak{P}}}{M_{\mathfrak{P}}} \cong \HHT{0}{H_{\mathfrak{P}}}{\Z} \cong \Z/\card{H_{\mathfrak{P}}}$. Therefore we obtain
\[
\prod_{H\leq G}\left(\frac{\hh{1}{H}{\idunits{K}{S}}}{\hh{2}{H}{\idunits{K}{S}}}\right)^{n_H} = \prod_{H\leq G}\prod_{\mathscr{p}\in S_k}\prod_{\mathfrak{p}|\mathscr{p}} \card{H_{\tilde{\mathfrak{P}}}}^{-n_H}
\]
where $S_k$ is the set of primes of $k$ which lies below primes in $S$. 

Now fix a prime $\mathscr{p}\in S_k$. Then the set of primes of $K^H$ dividing $\mathscr{p}$ is in one-to-one correspondence with the double cosets $H\backslash G\slash G_{\mathfrak{P}_*}$ where $\mathfrak{P}_*$ is a prime above $\mathscr{p}$ in $K$. Moreover $H_{\tilde{\mathfrak{P}}}=G_{\tilde{\mathfrak{P}}}\cap H$. We deduce that
\[
\prod_{H\leq G}\left(\frac{\hh{1}{H}{\idunits{K}{S}}}{\hh{2}{H}{\idunits{K}{S}}}\right)^{n_H} = \prod_{H\leq G}\prod_{\mathscr{p}\in S_k} \prod_{\tilde{\mathfrak{P}}\in H\backslash G/G_{\mathfrak{P}_*}}\card{G_{\tilde{\mathfrak{P}}}\cap H}^{-n_H}.
\]
The right-hand term of the above equality is equal to $\rc{\Theta}{\Z[S]}$ by \cite[Example 2.19]{DokchitserDokchitser2009}.
\end{proof}

\begin{lemma}\label{lemma:rczh2c}
For any Brauer relation $\Theta$ in $G$ we have
\[
\rc{\Theta}{\Z} = \prod_{H\leq G}\hh{2}{H}{\idCl{K}}^{-n_H}.
\]
\end{lemma}
\begin{proof}
Choosing the pairing defined by $\langle 1, 1\rangle = 1$ on $\Z^H = \Z$, one sees immediately that
\[
\rc{\Theta}{\Z} = \prod_{H\leq G} \card{H}^{-n_H}.
\]
On the other hand by global class field theory 
\[
\HH{2}{H}{\idCl{K}}\cong \HHT{0}{H}{\Z} = \Z/\card{H}\Z
\]
completing the proof.
\end{proof}

The regulator constants of the class group $\rc{\Theta}{\Cl{K}{S}}$ is immediately given by the following proposition, which is a version of Chevalley's ambiguous class number formula suited for our purposes. 

\begin{proposition}[Ambiguous class number formula]\label{prop:acnf}
The following formula holds for any subgroup $H$ of $G$
\[
\card{\left(\Cl{K}{S}\right)^H}=\cl{K^H}{S}\card{\Coker\left(\HH{2}{H}{\units{K}{S}}\to\HH{2}{H}{\idunits{K}{S}}\right)} \prod_{i=1}^2\left(\frac{\hh{i}{H}{\units{K}{S}}}{\hh{i}{H}{\idunits{K}{S}}}\right)^{(-1)^i}.
\]
\end{proposition}
\begin{proof}
The inclusion $K^H\hookrightarrow K$ induces a commutative diagram with exact rows
\[
\begin{tikzcd}
0\arrow[r]& \Qq{K^H}{S} \arrow[d]\arrow[r]& \idCl{K^H} \arrow[d]\arrow[r]&\Cl{K^H}{S} \arrow[d]\arrow[r]&0\\
0\arrow[r]& \left(\Qq{K}{S}\right)^H\arrow[r]& \left(\idCl{K}\right)^H\arrow[r]&\left(\Cl{K}{S}\right)^H \arrow[r]&\HH{1}{H}{\Qq{K}{S}}\arrow[r]&0
\end{tikzcd}
\]
Surjectivity of last map on the bottom row follows from $\HH{1}{H}{\idCl{K}}=0$ by class field theory. Moreover the central vertical arrow is an isomorphism. The snake lemma therefore gives an exact sequence
\begin{equation}\label{eq:amb}
0\to \left(\Qq{K}{S}\right)^H/\Qq{K^H}{S} \to \Cl{K^H}{S} \to \left(\Cl{K}{S}\right)^H\to \HH{1}{H}{\Qq{K}{S}}\to 0.
\end{equation}
We also have a commutative diagram with exact rows
\[
\begin{tikzcd}
0\arrow[r]& \units{K^H}{S} \arrow[d]\arrow[r]& \idunits{K^H}{S} \arrow[d]\arrow[r]&\Qq{K^H}{S} \arrow[d]\arrow[r]&0\\
0\arrow[r]& (\units{K}{S})^H\arrow[r]& \left(\idunits{K}{S}\right)^H\arrow[r]&\left(\Qq{K}{S}\right)^H \arrow[r]&\HH{1}{H}{\units{K}{S}}\arrow[r]&\HH{1}{H}{\idunits{K}{S}}
\end{tikzcd}
\]
The central vertical arrow is an isomorphism so we get an isomorphism
\[
\left(\Qq{K}{S}\right)^H/\Qq{K^H}{S} = \Ker\left(\HH{1}{H}{\units{K}{S}}\to\HH{1}{H}{\idunits{K}{S}}\right).
\]
In particular, by continuing the long exact cohomology sequence whose beginning is the bottom row of the above diagram, we get
\[
\card{\left(\Qq{K}{S}\right)^H/\Qq{K^H}{S}} =\hh{1}{H}{\Qq{S}{K}}\card{\Coker\left(\HH{2}{H}{\units{K}{S}}\to\HH{2}{H}{\idunits{K}{S}}\right)} \prod_{i=1}^{2}\left(\frac{\hh{i}{H}{\units{K}{S}}}{\hh{i}{H}{\idunits{K}{S}}}\right)^{(-1)^{i+1}} .
\]
Taking orders in \eqref{eq:amb} and using the above equality, we get the statement.
\end{proof}

Collecting the results obtained so far, one obtains the following formula relating the regulator constants of the group of units with class numbers.

\begin{theorem}\label{thm:rccln}
For any Brauer relation $\Theta$ in $G$ we have
\[
\rc{\Theta}{\units{K}{S}} = \rc{\Theta}{\Z[S]}^{-1}\rc{\Theta}{\Z} \prod_{H\leq G}\left(\cl{K}{S}\right)^{-2n_H}.
\]
\end{theorem}
\begin{proof}
The results obtained so far in this section allow us to write
\begin{align*}
\rc{\Theta}{\Cl{K}{S}}\left(\frac{\kd{\Theta}{\iota}}{\kd{\Theta}{\eta}}\right)^2 &=\prod_{H\leq G}\left(\cl{K^H}{S}\hh{2}{H}{\idCl{K}} \frac{\hh{2}{H}{\units{K}{S}}}{\hh{1}{H}{\units{K}{S}}}\frac{\hh{1}{H}{\idunits{K}{S}}}{\hh{2}{H}{\idunits{K}{S}}}\right)^{-2n_H} \quad\text{(Prop. \ref{prop:acnf} and Cor. \ref{cor:psiratio})} \\
&=\rc{\Theta}{\Z}^2\rc{\Theta}{\Z[S]}^{-2}\prod_{H\leq G}\left(\cl{K^H}{S} \frac{\hh{2}{H}{\units{K}{S}}}{\hh{1}{H}{\units{K}{S}}}\right)^{-2n_H} \quad\text{(Lemmas \ref{lemma:rczs} and \ref{lemma:rczh2c}).}
\end{align*}
Plugging the above equality in the formula of Proposition \ref{prop:firststep} and using Lemma \ref{lemma:rcws}, we get the statement.
\end{proof}

By an elementary linear algebra computation one can also express the regulator constant of the units in terms of a ratio of regulators and orders of the groups of roots of unity, as shown by Kani and also Bartel. Recall from the introduction that for a subextension $F/k$ of $K/k$, $\reg{F}{S}$ denotes the regulator of the $S_{F}$-units of $F$.

\begin{theorem}\label{thm:rcreg}
For any Brauer relation $\Theta$ in $G$ we have
\[
\rc{\Theta}{\units{K}{S}} = \rc{\Theta}{\roots{K}}\rc{\Theta}{\Z[S]}^{-1}\rc{\Theta}{\Z} \prod_{H\leq G}\left(\reg{K^H}{S}\right) ^{2n_H}.
\]
\end{theorem}
\begin{proof}
See \cite[Corollary 8.3 and Lemma 8.4]{Kani1994}. The result can also be deduced from  \cite[Proposition 2.15]{Bartel2012}, together with the fact that, by Proposition \ref{prop:rces}, 
\[
\rc{\Theta}{\units{K}{S}} = \rc{\Theta}{\mt{\units{K}{S}}}\rc{\Theta}{\roots{K}} \kd{\Theta}{p}^2
\]
where $p:\units{K}{S} \to \mt{\units{K}{S}}$ is the standard projection. 
\end{proof}

Putting together Theorem \ref{thm:rccln} and \ref{thm:rcreg}, we obtain the desired relation between class numbers, regulators and orders of the groups of roots of unity (recall that $\rc{\Theta}{\roots{K}} = \prod_{H\leq G}\card{\roots{K^H}}^{-2n_H}$ by definition).
\begin{theorem}
For any Brauer relation $\Theta=\sum_{H\leq G} n_H H$ in $G$ we have
\[
\prod_{H\leq G}\left(\cl{K^H}{S}\right)^{n_H}= \prod_{H\leq G}\left(\frac{\card{\roots{K^H}}}{\reg{K^H}{S}}\right)^{n_H}.
\]
\end{theorem}

\printbibliography
\end{document}